\documentclass[12pt, reqno]{amsart}
\usepackage{amsmath, amsthm, amscd, amsfonts, amssymb, graphicx, color, booktabs, multirow, rotating}
\newtheorem{theorem}{Theorem}[section]

\newtheorem{proposition}[theorem]{Proposition}

\theoremstyle{definition}
\newtheorem{definition}[theorem]{Definition}
\newtheorem{example}[theorem]{Example}

\theoremstyle{remark}
\newtheorem{remark}[theorem]{Remark}
\numberwithin{equation}{section}

\usepackage{listings}

\begin{document}
\setcounter{page}{1}

\title{The Rank of the Cartier operator on Picard Curves}


%    author one information
\author{Vahid Nourozi And Farhad Rahmati$^{*}$}
%\address{Faculty of Mathematics and Computer Science, Amirkabir University of Technology\\
%(Tehran Polytechnic), 424 Hafez Ave., Tehran 15914, Iran}
%\curraddr{Farhad Rahmati}
%\email{nourozi.v@gmail.com; nourozi@aut.ac.ir}
\address{Faculty of Mathematics and Computer Science, Amirkabir University of Technology\\
(Tehran Polytechnic), 424 Hafez Ave., Tehran 15914, Iran}
\email{nourozi@aut.ac.ir; nourozi.v@gmail.com}
\email{frahmati@aut.ac.ir}

\thanks{*\, Corresponding author}

%    author two information
%\author{}
%\address{}
%\curraddr{Farhad Rahmati}
%\email{}
%\thanks{}

%    \subjclass and \keywords are not used by JAG.

%\date{March 2018}

%\dedicatory{}

%    "Communicated by" -- provide editor's name; required.
%\commby{m}

%    Abstract is required.
\begin{abstract}
%The $a$-number is an invariant of the isomorphism class of the $p$-torsion group scheme. In this paper, we compute $a$-number of Picard  curves for $p \equiv 1$ mod $3$ and, for $p \equiv 1$ mod $3$ , using the action of the Cartier operator on $H^0(\mathcal{X},\Omega^1)$.

For an algebraic curve $\mathcal{X}$ defined over an algebraically closed field of characteristic $p > 0$, $a$-number $a(\mathcal{X})$ is the dimension of the space of exact holomorphic differentials on $\mathcal{X}$. We computed the $a$-number for a family of certain Picard curves using the action of the Cartier operator on $H^0(\mathcal{X},\Omega^1)$.

%For $q=p^s$, the curve $\mathcal{X}$ is given by the equation $y^{q}+y=x^{\frac{q+1}{2}}$ with characteristic $p$, and for $q=2^s$, the curve $\mathcal{Y}$ is given by the equation $\sum_{t=1}^s y^{q/2^t}=x^{q+1}$  with characteristic two. The $a$-number is an invariant of the isomorphism class of the $p$-torsion group scheme. In this paper, we compute a closed formula for the $a$-number of $\mathcal{X}$ (resp. $\mathcal{Y}$) using the action of the Cartier operator on $H^0(\mathcal{X}, \Omega^1)$ (resp. $H^0(\mathcal{Y}, \Omega^1)$).
\end{abstract}

\maketitle
\section{Introduction}

Let $\mathcal{X}$ be a geometrically irreducible, projective, and non-singular algebraic curve defined over the finite field $\mathbb{F}_{\ell}$ of order $\ell$. Let $\mathcal{X}(\mathbb{F}_{\ell})$ denote the set of $\mathbb{F}_{\ell}$-rational points for $\mathcal{X}$. In studying curves over finite fields, determining the size of  $\mathcal{X}(\mathbb{F}_{\ell})$ is a fundamental problem. The exact primary result here is the Hasse-Weil bound, which asserts that
$$\mid \# \mathcal{X}(\mathbb{F}_{\ell}) - (\ell +1) \mid \leq 2g \sqrt{\ell},$$
where $g = g(\mathcal{X})$ is the genus of $\mathcal{X}$.

In this paper, we consider Picard curves of genus $3$  over $\mathbb{F}_{q^2}$ of characteristic $p > 3$. Let $\mathcal{X}$ be a Picard curve. Then $\mathcal{X}$ can be defined by an affine equation of the form
 \begin{equation}\label{xxx}
y^3 =f(x),
 \end{equation}

where $f(x)$ is a polynomial of degree $4$, without multiple roots.
It is easy to see that the maximal curve $\mathcal{X}$ is super-singular because all the slopes of its Newton polygon are equal to $1/2$. This implies that the Jacobian $X:=\mbox{Jac}(\mathcal{X})$ has no $p$-torsion points over $\bar{\mathbb{F}}_{p}$.
 %There are many questions for $\mbox{Jac}(\mathcal{X})$.
A relevant invariant of the $p$-torsion group scheme of the Jacobian of the curve is the $a$-number.

 Consider multiplication by $p$-morphism $[p]: X \rightarrow X$, a finite flat morphism of degree $p^{2g}$.  These factors were $[p]=V \circ F$. Here, $F: X \rightarrow X^{(p)}$   is the relative Frobenius morphism originating from the $p$-power map on the sheaf structure, and the Verschiebung morphism  $V: X^{(p)} \rightarrow X$ is the dual of $F$. The multiplication-by-$p$ kernel on $X$ is defined as $X[p]$.
An important invariant  is the $a$-number $a(\mathcal{X})$ of curve $\mathcal{X}$ defined by
$$a(\mathcal{X})=\mbox{dim}_{\mathbb{\bar{F}}_p} \mbox{Hom}(\alpha_{p}, X[p]),$$
where $\alpha_p$ is the kernel of the Frobenius endomorphism in the group scheme $\mbox{Spec}(k[X]/(X^p))$. Another definition for the $a$-number is
$$ a(\mathcal{X}) = \mbox{dim}_{\mathbb{F}_p}(\mbox{Ker}(F) \cap \mbox{Ker}(V)).$$
%be the group scheme $\mbox{Spec}(k[X]/(X^p))$ with co-multiplication given by
%$$X\rightarrow 1 \otimes X+X \otimes 1.$$
%Write $X = J (A)$, the Jacobian variety of a curve $A$ (non-singular, projective, geometrically irreducible, algebraic) defined over $k$.
%The  $a$-number $a(\mathcal{X})$ defined to be the dimension of the vector space $\mbox{Hom}(\alpha_p, A)$. And we can write $a(\mathcal{X})$ instead of $a_{J(\mathcal{X})}$ that refers to the $a$-number of $\mathcal{X}$. see \cite{suz}

 %Let $\mathcal{A}$ be an abelian variety defiend over  $k$. Let $\alpha_p$ be the group scheme $%\mbox{Spec}(k[X]/(X^p))$ with co-multiplication given by
%$$X\rightarrow 1 \otimes X+X \otimes 1.$$
%The group $\mbox{Hom}(\alpha_p, A)$ can be considered as $k$-vector space since $\mbox{End}(\alpha_p)=k$.
%The  $a$-number $a(\mathcal{A} )$ defined to be the dimension of the vector space $\mbox{Hom}(\alpha_p, A)$.
%Normally we can define  the $a$-number $a(\mathcal{X})$ of $\mathcal{X}$ as  the $a$-number of its Jacobian variety $\mathcal{J}_{\mathcal{X}}$. As a matter of fact, the $a$-number of a curve is a  birational invariant which can be defined the dimension of the space of exact holomorphic differentials.

 A few results on the rank of the Cartier operator (especially $a$-number) of curves are introduced by Kodama and Washio \cite{13}, González \cite{8}, Pries and Weir \cite{17}, Yui \cite{Yui}. Besides that, Vahid talked about the rank of the Cartier of the maximal curves in \cite{esfahan}, the hyperelliptic curve in \cite{shiraz}, and the maximal function fields of $\mathbb{F}_{q^2}$ in \cite{misori, behrooz}. Vahid also wrote about these topics in his \cite{phd} dissertation.

 In Section \ref{202}, we prove that the $a$-number of the  Picard curves  with Equation (\ref{xxx}) is $0$ for $p \equiv 1$ mod $3$, and is $1$ for $p \equiv 2$ mod $3$ , see Theorem \ref{thex}. The proofs directly use the Cartier operator action on $H^0(\mathcal{X}, \Omega^1)$. Finally, we provide an example of the $a$-number of Picard curves obtained from the Hasse-Witt matrix.

\section{The Cartier operator}
Let $k$ be an algebraically closed field of characteristic $p>0$.
Let $\mathcal{X}$ be a
curve defined over $k$.
The Cartier operator is a $1/p$-linear operator acting on the sheaf $\Omega^1:=\Omega^1_{\mathcal{X}}$ of differential forms on $\mathcal{X}$ with positive characteristics, $p>0$.

 Let $K=k(\mathcal{X})$ be the function field of curve $\mathcal{X}$ of genus $g$ defined over  $k$. The separating variable for $K$ is an element $x \in K \setminus K^p$. We can write any function $f \in K$ uniquely in form
 $$f=f^p_0 + f^p_1x +\cdots + f^p_{p-1}x^{p-1},$$
 where $f_i \in K$, for $i = 0, \cdots , p -1$.

\begin{definition}
  (Cartier Operator). Let $\omega \in  \Omega_{K/K_p}$. There exist $f_0,\cdots, f_{p-1}$ such
that $\omega = (f^p_0 + f^p_1x +\cdots + f^p_{p-1}x^{p-1})dx$. The Cartier operator $\mathfrak{C}$ is defined by
$$\mathfrak{C}(\omega) := f_{p-1}dx.$$
This definition does not depend on the choice of $x$ (see \cite[Proposition 1]{100}).
\end{definition}
We refer the reader to \cite{100,30,40,150} for proofs of the following statements.
%\begin{proposition}
 % (Local properties of $\mathcal{C}$). Let $P$ be a place of $F$ . For all $\omega \in  \Omega_{F/F_q^l}$,
%  \begin{itemize}
 %   \item [1.]  $v_P (\omega) \geq 0 \Rightarrow v_P (\mathcal{C}(\omega)) \geq 0$;
 %   \item [2.]  $v_P (\omega) \leq -2 \Rightarrow v_P (\mathcal{C} (\omega)) > v_P (\omega)$;
 %   \item [3.]  $v_P (\omega) = -1 \Rightarrow v_P (\mathcal{C} (\omega)) = -1$;
 %   \item [4.]  $res_P (\mathcal{C} (\omega)) = res_P (\omega)^{1/p}$.
%  \end{itemize}
%\end{proposition}

\begin{proposition}
  (Global Properties of $\mathfrak{C}$). For all $\omega \in  \Omega_{K/K_q}$ and all $f \in K$,

  \begin{itemize}
    \item $\mathfrak{C}(f^p\omega) = f\mathfrak{C}(\omega)$;
    \item $\mathfrak{C}(\omega) = 0 \Leftrightarrow \exists h \in K, \omega = dh$;
    \item $\mathfrak{C}(\omega) = \omega \Leftrightarrow \exists h \in K, \omega = dh/h$.
  \end{itemize}
\end{proposition}

\begin{remark}\label{remark}
Moreover, one can easily show
that
\begin{equation*}
\mathfrak{C}(x^j dx) = \left\{
\begin{array}{ccc}
    0 & \mbox{if}&  \hspace{.4cm} p \nmid j+1  \\
    x^{s-1}dx & \mbox{if} &  \hspace{.4cm} j+1=ps.
\end{array} \right.
\end{equation*}
\end{remark}

%A differential $\omega$ is holomorphic if $div(\omega)$ is effective. The set $H^0(\mathcal{X}, \Omega^1)$ of holomorphic differentials is a $g$-dimensional $K$-vector subspace of $\Omega^1$ such that $\mathcal{C}(H^0(\mathcal{X}, \Omega^1)) \subseteq H^0(\mathcal{X}, \Omega^1)$.

%The dimension $a(\mathcal{X})$ of the kernel of $\mathcal{C}$ (or equivalently, the dimension of the space of exact holomorphic differentials on $\mathcal{X}$) is the $a$-number of $\mathcal{X}$.

%Let $\mathcal{B}=\{\omega_1,\cdots,\omega_g\}$ be a basis of $H^0(\mathcal{X}, \Omega^1)$. Then for any $\omega \in H^0(\mathcal{X}, \Omega^1)$,
%$$\mathcal{C}(\omega) =\sum_{i=1}^{g}a_{i,j}\omega_i.$$
%The matrix $A(\mathcal{X}) =(a^{1/p}_{ij})$ is the Hasse–Witt(or Cartier–Manin) matrix of $\mathcal{X}$. The $a$-number $a(\mathcal{X})$ is the co-rank of $A(\mathcal{X})$ (or, equivalently, of $A^p(\mathcal{X}) =(a_{ij})$). see \cite{maria}.

%\begin{remark}\label{r3.1}
%\rm{ Moreover, one can easily show
%that

%\[\mathscr{C}(x^j dx) = \left\{
%\begin{array}{ccc}
%0 & \mbox{if}&  \hspace{.4cm} p \nmid j+1 \\
%x^{s-1}dx & \mbox{if} &  \hspace{.4cm} j+1=ps. \\
%\end{array}\right.\]}
%\end{remark}

If $\mbox{div}(\omega)$ is effective, then differential $\omega$ is holomorphic. The set $H^0(\mathcal{X}, \Omega ^1)$ of holomorphic differentials is a $g$-dimensional $k$-vector subspace of $\Omega^1$, such that $\mathfrak{C}(H^0(\mathcal{X}, \Omega^1)) \subseteq H^0(\mathcal{X}, \Omega^1)$. If $\mathcal{X}$ is a curve, then the $a$-number of $\mathcal{X}$ equals the dimension of the kernel of the Cartier operator $H^0(\mathcal{X}, \Omega^1)$ (or equivalently, the dimension of the space of the exact holomorphic differentials on $\mathcal{X}$) (see \cite[5.2.8]{14}).
% Let $\mathcal{X}$ be a plane model of $\mathcal{X}$ given by an affined equation $\mathcal{X} : F(X, Y ) = 0$ with an irreducible polynomial $F \in K[X, Y ]$ of degree $n > 3$. Then $K(\mathcal{X}) = K(x, y)$ with $F(x, y) = 0$. If $\mathcal{X}$ has only ordinary singularities, its canonical adjoints are the curves of formal degree $n - 3$ with at least an $(r - 1)$-fold point at every $r$-fold point of $\mathcal{X}$; see \cite{1111}.

The following theorem is based on Gorenstein; see \cite[Theorem 12]{9}.
\begin{theorem}\label{2.2}
  A differential $\omega \in \Omega^1$ is holomorphic if and only if it has the form $(h(x, y)/F_y)dx$, where $H:h(X, Y) =0$ is a canonical adjoint.
\end{theorem}

\begin{theorem}\cite{maria}\label{2.3}
  With the above assumptions,
\begin{equation*}
\mathfrak{C}(h\dfrac{dx}{F_{y}}) = (\dfrac{\partial^{2p-2}}{\partial x^{p-1}\partial y^{p-1}}(F^{p-1}h))^{\frac{1}{p}}\dfrac{dx}{F_{y}}
\end{equation*}
for any $h \in K(\mathcal{X})$.
\end{theorem}

The differential operator $\nabla$ is defined by
$$\nabla = \dfrac{\partial^{2p-2}}{\partial x^{p-1} \partial y^{p-1}},$$
has the property
\begin{equation}\label{123321}
  \nabla(\sum_{i,j} c_{i,j}X^iY^j)= \sum_{i,j} c_{ip+p-1,jp+p-1}X^{ip}Y^{jp}.
\end{equation}

\section{The $a$-number of  Picard Curves}\label{202}

This section considers the Picard curves over a finite field with $q^2$ elements of characteristic $p > 3$. Let $\mathcal{X}$ be the Picard curve of genus $3$ over $\mathbb{F}_{q^2}$. Then $\mathcal{X}$ can be defined by an affine equation of the form
\begin{equation}
y^3=f(x),
\end{equation}
where $f(x)$ is a polynomial of degree $4$ without multiple roots. From Theorem \ref{2.2}, one can find a basis for the space $H^0(\mathcal{X}, \Omega^1)$ of holomorphic differentials on
$\mathcal{X}$, namely

$$\mathcal{B} = \Bigg\{ z_i=\dfrac{h_idx}{f_y} \Bigg\},$$
where $h_1 = 3, h_2 = 3x$ and $h_3 = 3y$ for $1 \leq i \leq 3$. Therefore, $\mathcal{B} = \{ \frac{dx}{y^2},\frac{xdx}{y^2},\frac{dx}{y}
\}$.

\begin{theorem}\label{thex}
Suppose that $p\not\equiv 0$ mod $3$ and $f(x)$ have a non-zero constant-coefficient, then
 \begin{itemize}
 \item[1.] If $p \equiv 1$ mod $3$, then $a(\mathcal{X})=0$.
 \item[2.] If $p \equiv 2$ mod $3$, then $a(\mathcal{X})=1$.
 \end{itemize}
\end{theorem}
\begin{proof}
\begin{itemize}
\item[1.] Let us compute the image of $\mathfrak{C}(\omega)$ for any $\omega \in \mathcal{B}$. Taking $p =3r+1$ and $r\in \mathbb{Z}$. Since $p -1 =3r$. We get $2p-2 = 6r$. So
  \begin{equation*}
  \begin{array}{ccccccc}
              \mathfrak{C}(xdx/y^2) &=& \mathfrak{C}(y^{2p-2}y^{-2p}xdx)  = y^{-2}\mathfrak{C}(xf(x)^{2r}dx)  \\
              &=& y^{-2}\sum_{i=0}^{2r}c_i \mathfrak{C}(x^{i+1}dx) =  c_rdx\neq 0.\\
             \end{array}
\end{equation*}
  \begin{equation*}
  \begin{array}{ccccccc}
              \mathfrak{C}(dx/y^2) &=& \mathfrak{C}(y^{2p-2}y^{-2p}dx)  = y^{-2}\mathfrak{C}(f(x)^{2r}dx)  \\
              &=& y^{-2}\sum_{i=0}^{2r}c_i \mathfrak{C}(x^{i}dx) =  c_rdx\neq 0.\\
             \end{array}
\end{equation*}
  Now we let $p -1 =3r$,
  \begin{equation*}
  \begin{array}{ccccccc}
              \mathfrak{C}(dx/y) &=& \mathfrak{C}(y^{p-1}y^{-p}dx)  = y^{-1}\mathfrak{C}(f(x)^{r}dx)  \\
              &=& y^{-1}\sum_{i=0}^{r}c_i \mathfrak{C}(x^{i}dx) =  c_rdx\neq 0.\\
             \end{array}
\end{equation*}
Thus, $a(\mathcal{X})=0$.
\item[2.] We assume that $p =3r+2$ and $r\in \mathbb{Z}$. We get $p-2=3r$. Hence
\begin{equation*}
  \begin{array}{ccccccc}
              \mathfrak{C}(xdx/y^2) &=& \mathfrak{C}(y^{p-2}y^{-p}xdx)  = y^{-1}\mathfrak{C}(xf(x)^{r}dx)  \\
              &=& y^{-1}\sum_{i=0}^{r}c_i \mathfrak{C}(x^{i+1}dx) =  c_rdx\neq 0.\\
             \end{array}
\end{equation*}
  \begin{equation*}
  \begin{array}{ccccccc}
              \mathfrak{C}(dx/y^2) &=& \mathfrak{C}(y^{p-2}y^{-p}dx)  = y^{-1}\mathfrak{C}(f(x)^{r}dx)  \\
              &=& y^{-1}\sum_{i=0}^{r}c_i \mathfrak{C}(x^{i}dx) =  c_rdx\neq 0.\\
             \end{array}
\end{equation*}
  Finally, from Remark \ref{remark} we have
  \begin{equation*}
  \begin{array}{ccccccc}
              \mathfrak{C}(dx/y) &=& \mathfrak{C}(y^{p-1}y^{-p}dx)  = y^{-1}\mathfrak{C}(f(x)^{\frac{3r-1}{3}}dx)  \\
              &=& y^{-1}\sum_{i=0}^{\frac{3r-1}{3}}c_i \mathfrak{C}(x^{i}dx) = 0.\\
             \end{array}
\end{equation*}
Therefore, $a(\mathcal{X})=1$.
\end{itemize}
 Where the coefficients $c_i \in k$ are obtained from the
 expansion
 $$y^{p-1}=f(x)^{(p-1)/2}= \sum_{i=0}^{N} c_i x^i \hspace{.7cm} ~\mbox{with}~ N=\frac{p-1}{2}.$$
\end{proof}

In \cite{picard}, Estrada Sarlabous showed that if $F^{p-1} =\sum_{i,j} c_{i,j}x^iy^j$, then the Hasss-Wit matrix of $\mathcal{B}$ is
\[H=\left( \begin{array}{cccc}
c_{p-1,p-1}& c_{2p-1,p-1} & c_{p-1,2p-1} \\
c_{p-2,p-1}&c_{2p-2,p-1} & c_{p-2,2p-1}\\
c_{p-1,p-2}& c_{2p-1,p-2} & c_{p-1,2p-2}
\end{array} \right).\]
\newline
Since $F = y^3 - f(x), c_{i,j} = 0$ for $j \not\equiv 0$ mod $3$. Thus, we have
\begin{itemize}
\item[1.] $p \equiv 1$ mod $3$. Then
\begin{equation}\label{hass1}
H=\left( \begin{array}{cccc}
c_{p-1,p-1}& c_{2p-1,p-1} & 0 \\
c_{p-2,p-1}&c_{2p-2,p-1} & 0\\
0& 0 & c_{p-1,2p-2}
\end{array} \right).
\end{equation}
where $c_{p-1,p-1}$, $c_{2p-1,p-1}$, $c_{p-2,p-1},c_{2p-2,p-1}$ are the respective coefficients of $x^{p-1}, x^{2p-1}$, $x^{p-2}$ and $x^{2p-2}$ in $f(x)^{(2p-2)/3}$ mod $p$. $c_{p-1,2p-2}$ is the coefficient of $x^{p-1}$ in $f(x)^{(p-1)/3}$ mod $p$.
\item[2.] $p \equiv 2$ mod $3$. Then
\begin{equation}\label{hass2}
H=\left( \begin{array}{cccc}
0&0 & c_{p-1,2p-1} \\
0&0 & c_{p-2,2p-1}\\
c_{p-1,p-2}& c_{2p-1,p-2} & 0
\end{array} \right).
\end{equation}
where $c_{p-1,2p-1}$ and $c_{p-2,2p-1}$ are the coefficients of $x^{p-1}$ and $x^{p-2}$ in $f(x)^{(p-2)/3}$ mod $p$, and $c_{p-1,p-2}$, and $c_{2p-1,p-2}$ are the coefficients of $x^{p-1}$ and $x^{2p-1}$ in $f(x)^{(2p-1)/3}$ mod $p$.
\end{itemize}
\begin{example}
Consider the Picard curves with the following equation:
$$y^3=x^4+1.$$
We calculate the a-number with the Hasse-Wit matrix of this curve with $p=5$ and $p=13$. Firstly, suppose that $p=5$ in this case from equation (\ref{hass2}), we have
\begin{equation*}
H=\left( \begin{array}{cccc}
0&0 & 1 \\
0&0 & 0\\
3& 0 & 0
\end{array} \right).
\end{equation*}
So $\mbox{rank}(H)=2$. From the definition of $a$-number, it is clear that $a(\mathcal{X})=g(\mathcal{X})-\mbox{rank}(H)$. Therefore, $a(\mathcal{X})=1$. Finally, assuming that $p=13$  in this case from equation (\ref{hass1}), we have
\begin{equation*}
H=\left( \begin{array}{cccc}
4& 0 & 0 \\
0&2 & 0\\
0& 0 & 4
\end{array} \right).
\end{equation*}
So $\mbox{rank}(H)=3$, hence $a(\mathcal{X})=0$. In addition, we achieve this result from MAGMA computation \cite{magma}.
\end{example}

\paragraph*{\textbf{Acknowledgements.}}
This paper was written while Vahid Nourozi was visiting Unicamp (Universidade Estadual de Campinas) supported by TWAS/Cnpq (Brazil) with fellowship number $314966/2018-8$. We are extremely grateful to the referee for their valuable comments and suggestions, which led us to find correct references for many assertions and improve the exposition.
\bibliographystyle{amsplain}

\end{document}